\documentclass[12pt,a4paper,twoside,reqno]{amsart}
\allowdisplaybreaks

\usepackage[all,cmtip]{xy} 
\usepackage[T1]{fontenc}
\usepackage{amsmath,amscd}
\usepackage{amssymb,amsfonts}
\usepackage[driver=pdftex,margin=3cm,heightrounded=true,centering]{geometry}
\usepackage{mathtools}
\usepackage{tensor}
\usepackage{url}
\usepackage[colorlinks=true,linkcolor=blue,citecolor=blue]{hyperref}
\usepackage{slashed}
\usepackage{comment}

\usepackage{graphicx}

\usepackage{amsthm}

\theoremstyle{plain}

\newtheorem{theorem}{Theorem}[section]
\newtheorem*{thm*}{Theorem}
\newtheorem{proposition}[theorem]{Proposition}
\newtheorem{lemma}[theorem]{Lemma}
\newtheorem{corollary}[theorem]{Corollary}

\theoremstyle{definition}
\newtheorem{definition}[theorem]{Definition}

\theoremstyle{remark} 

\numberwithin{equation}{section}

\newcommand{\alpheqn}[1][\relax]{
     \refstepcounter{equation}
     \if#1\relax \relax
       \else \label{#1}
     \fi  
     \setcounter{saveeqn}{\value{equation}}%
    \setcounter{equation}{0}%
    \renewcommand{\theequation}{\thealphequation}}
\newcommand{\reseteqn}{\setcounter{equation}{\value{saveeqn}}%
     \renewcommand{\theequation}{\thearabicequation}}

\providecommand{\mathscr}{\mathcal} 
\IfFileExists{mathrsfs.sty}{
\usepackage{mathrsfs}}         
   {
     \IfFileExists{eucal.sty}{
        \usepackage[mathscr]{eucal}} 
   {
   }
}

\newcommand{\vertiii}[1]{{\left\vert\kern-0.25ex\left\vert\kern-0.25ex\left\vert #1 
    \right\vert\kern-0.25ex\right\vert\kern-0.25ex\right\vert}}
\newcommand{\Bvert}[1]{{\Big\vert\kern-0.25ex\Big\vert\kern-0.25ex\Big\vert #1 
    \Big\vert\kern-0.25ex\Big\vert\kern-0.25ex\Big\vert}}
\newcommand{\bvert}[1]{{\big\vert\kern-0.25ex\big\vert\kern-0.25ex\big\vert #1 
    \big\vert\kern-0.25ex\big\vert\kern-0.25ex\big\vert}}
\newcommand{\nvert}[1]{{\vert\kern-0.25ex\vert\kern-0.25ex\vert #1 
    \vert\kern-0.25ex\vert\kern-0.25ex\vert}}

\renewcommand{\leq}{\leqslant}
\renewcommand{\geq}{\geqslant}

\newcommand{\cd}{\cdot}

\newcommand{\op}{\oplus}

\newcommand{\ti}{\times}

\newcommand{\al}{\alpha}

\newcommand{\ga}{\gamma}

\newcommand{\de}{\delta}

\newcommand{\ep}{\varepsilon}

\newcommand{\io}{\iota}

\newcommand{\si}{\sigma}

\newcommand{\te}{\theta}

\newcommand{\ov}{\overline}
\newcommand{\C}[1]{\mathcal{#1}}

\newcommand{\T}[1]{\textup{#1}}

\newcommand{\B}[1]{\mathbb{#1}}

\newcommand{\fork}[2]{\left\{ \begin{array}{#1} #2 \end{array} \right.}

\newcommand{\su}{\subseteq}

\newcommand{\q}{\qquad}

\newcommand{\inn}[1]{\langle #1 \rangle}

\newcommand{\sem}{\setminus}

\usepackage{amssymb}
\usepackage{amsthm}
\usepackage{thmtools}
\usepackage{mathtools}
\usepackage{derivative}

\newcommand{\N}{\mathbb{N}}
\newcommand{\Z}{\mathbb{Z}}
\renewcommand{\O}{\mathcal{O}}
\renewcommand{\epsilon}{\varepsilon}
\renewcommand{\phi}{\varphi}
\newcommand{\X}{\mathcal{X}}
\newcommand{\compdist}{\mathrm{dist}^s}
\newcommand{\Y}{\mathcal{Y}}

\newcommand{\Hil}{\mathcal{H}}
\newcommand{\dif}{\mathrm{d}}

\newcommand{\diam}{\mathrm{diam}}

\DeclareMathOperator{\Lip}{Lip}

\DeclarePairedDelimiter\abs{\lvert}{\rvert}
\DeclarePairedDelimiter\norm{\lVert}{\rVert}

\usepackage{todonotes}
\usepackage{cleveref}
\usepackage[shortlabels]{enumitem}
\usepackage{subfiles}

\makeatletter
\@namedef{subjclassname@2020}{%
  \textup{2020} Mathematics Subject Classification}
\makeatother

\begin{document}
\title[Quantum deformations of the arc length metric]{Quantum deformations of the arc length metric}


\author{Rasmus Hauge Hansen and Jens Kaad}

\address{Department of Mathematics and Computer Science,
The University of Southern Denmark,
Campusvej 55, DK-5230 Odense M,
Denmark}



%
%
%

\begin{abstract}
  We investigate a $q$-deformation of the arc length metric on the unit circle. This $q$-deformation arises naturally from the Dirac operator by replacing the standard integers by their $q$-deformed analogues. Nonetheless, we show that the corresponding metric structure only makes sense at the level of quantum metric spaces as introduced by Marc Rieffel. This means that the quantum metric we obtain on the continuous functions on the circle does not arise from a classical metric on the circle. In the special case where $q$ equals one we recover the usual arc length metric and we show that our family of quantum metric spaces depend continuously on the deformation parameter with respect to David Kerr's complete Gromov-Hausdorff distance.
%
\end{abstract}

\maketitle
\tableofcontents

\section{Introduction}
The theory of compact quantum metric spaces was pioneered by Rieffel in a series of papers, \cite{Rie:MSA,Rie:MSS,Rie:GHD}. The inspiration and motivation come from the rich theory of classical compact metric spaces in combination with the large program on noncommutative geometry as initiated by Connes, \cite{Con:CMS,Con:NCG,Con:GCM}. Most prominently, Rieffel's theory admits an interesting analogue of the classical Gromov-Hausdorff distance which is referred to as the quantum Gromov-Hausdorff distance and permits us to treat convergence questions for compact quantum metric spaces in a systematic way, \cite{Rie:GHD}. By now, the theory of compact quantum metric spaces has been developed considerably due to contributions of several authors, see the incomplete list \cite{Ker:MQG,Li:TDC,OzRi:HGC,ChIv:STA,Lat:QGH,CoSu:STN,AgKaKy:PSC,LeSu:GHC} for some of the highlights. 

We are in this text concerned with Kerr's complete Gromov-Hausdorff distance $\compdist$ which is an enhanced version of Rieffel's quantum Gromov-Hausdorff distance taylored to deal with operator systems, see \cite{Ker:MQG}. For other refinements of the quantum Gromov-Hausdorff distance we refer to the extensive work of Latr\'emoli\`ere, see e.g. \cite{Lat:QGH,Lat:GHP}.

To explain the contributions of the present paper we recall the definition of a compact quantum metric space. In the operator system approach, a compact quantum metric space is given by an operator system $\C X$ and a seminorm $L : \C X \to [0,\infty)$ which is supposed to be $*$-invariant and vanish on the unit. The key condition is then that the associated \emph{Monge-Kantorovich metric} on the state space $S(\C X)$ metrizes the weak-$\ast$\ topology. If $\C X$ sits densely inside a unital $C^*$-algebra $A$, then we say that $(\C X,L)$ is a quantum metric on $A$.

  Let us now fix a compact Hausdorff space $M$ together with the unital $C^*$-algebra $C(M)$ of continuous complex-valued functions on $M$ equipped with the supremum norm $\| \cd \|_\infty$. An intriguing observation in the commutative setting is that there are examples of quantum metrics $(\C X,L)$ on $C(M)$ which do not come from a metric on $M$. This means that it is impossible to find a metric on $M$ such that $L(f)$ agrees with the Lipschitz constant for all $f \in \C X$. For a thorough discussion of these matters we refer to Rieffel's paper \cite{Rie:MSS} and in particular \cite[Theorem 4.2 and Theorem 8.1]{Rie:MSS}. Remark however that the results in \cite{Rie:MSS} are confined to order unit spaces and hence, in the commutative setting, they are concerned with the continuous real-valued functions $C(M,\B R)$ as opposed to the continuous complex-valued functions $C(M)$.  


In this paper, we investigate the simple setting where the compact Hausdorff space $M$ agrees with the unit circle $S^1 \su \B C$. Let us apply the notation $\C O(S^1) \su C(S^1)$ for the smallest unital $*$-subalgebra containing the inclusion $z : S^1 \to \B C$. We present a whole family $\big( (\Lip_q(S^1), L_q)\big)_{q \in (0,1]}$ of quantum metrics on $C(S^1)$ such that $\C O(S^1) \su \Lip_q(S^1)$ for all $q \in (0,1]$. This family moreover has the following properties: 
  \begin{itemize}
  \item Let $q = 1$. The operator system $\Lip_1(S^1)$ agrees with the Lipschitz functions on $S^1$ equipped with the usual arc length metric and for $f \in \Lip_1(S^1)$ it holds that $L_1(f)$ is the Lipschitz constant. The compact quantum metric space $(\Lip_1(S^1), L_1)$ therefore comes from the arc length metric on $S^1$.
  \item Let $q \in (0,1)$ and let $\C X \su \Lip_q(S^1)$ be an operator subsystem with $\C O(S^1) \su \C X$. The quantum metric $(\C X, L_q)$ on $C(S^1)$ does not come from a metric on $S^1$. In fact, it is impossible to find a constant $C > 0$ such that
    \[
L_q(f \cd g) \leq C \cd \big( L_q(f) \cd \| g \|_\infty + \| f \|_\infty \cd L_q(g) \big) \q \T{for all } f,g \in \C O(S^1) .
    \]
In particular, $\big( \C O(S^1), L_q\big)$ does not satisfy the Leibniz inequality. 
  \item Our family of compact quantum metric spaces is continuous with respect to Kerr's complete Gromov-Hausdorff distance. Thus, for every $q_0 \in (0,1]$, we have that 
    \[
\lim_{q \to q_0} \compdist\big( (\Lip_q(S^1), L_q), ( \Lip_{q_0}(S^1), L_{q_0} ) \big) = 0 .
    \]
  \item Let $q \in (0,1]$. The complete Gromov-Hausdorff distance between the minimal and the maximal quantum metrics on $C(S^1)$ is equal to zero. This means that $\compdist\big( (\C O(S^1),L_q), (\Lip_q(S^1),L_q) \big) = 0$.
  \end{itemize}

  The above properties are achieved by applying the theory of Schur multipliers, see e.g. \cite{Pis:SPC}. More precisely, many of our constructions and estimates are related to results from \cite{JKDK_QSU2} where similar structures control part of a family of quantum metrics on quantum $SU(2)$.  
  
Let us end this introduction by a brief explanation of the basic question which motivates our work. Consider the Dirac operator $-i \frac{d}{d\theta} : C^\infty(S^1) \to C^\infty(S^1)$ for the unit circle and observe that
\begin{equation}\label{eq:dirac}
-i \frac{d}{d\theta}(z^n) = n \cd z^n \q \T{for all } n \in \B Z .
\end{equation}
The reconstruction theorem in noncommutative geometry entails that one may, in an appropriate sense, recover the unit circle from this Dirac operator, see \cite{Con:SCM}. For $q \in (0,1)$ one could therefore ask what happens to the geometry of the unit circle, if we replace the integer $n$ on the right hand side of \eqref{eq:dirac} by the corresponding $q$-integer $[n]_q = \frac{q^n - q^{-n}}{q - q^{-1}}$. In this paper, we investigate how this replacement affects the (quantum) metric properties of the compact Hausdorff space $S^1$.

\section{Notation and conventions}\label{s:conventions}
For a Hilbert space \(\Hil\), the notation \(\B B(\Hil)\) refers to the unital \(C^*\)-algebra of bounded operators on \(\Hil\) equipped with the operator norm $\| \cd \|_\infty : \B B(\Hil) \to [0,\infty)$. We use the convention that the inner product on the Hilbert space \(\Hil\) is linear in the second leg.
  
Apply the notation $\ell^2(\B Z)$ for the separable Hilbert space of $\ell^2$-sequences indexed by the integers. The standard orthonormal basis for $\ell^2(\B Z)$ is denoted by \((e_j)_{j\in\Z}\). The bilateral shift is the unitary operator $U : \ell^2(\B Z) \to \ell^2(\B Z)$ determined by the rule $U(e_j) = e_{j+1}$. As is customary we identify the unital $C^*$-algebra of continuous functions on the unit circle, $C(S^1)$, with the unital $C^*$-subalgebra $C^*(U) \su \B B(\ell^2(\B Z))$ generated by $U$. This identification happens via the map sending the inclusion $z : S^1 \to \B C$ to the bilateral shift $U$. 

For a finite set $F$, the notation $|F| \in \B N_0$ refers to the number of elements in $F$.

In this text, we work with concrete operator systems, meaning that for us, an operator system $\C X$ sits as a (not necessarily closed) subspace of a fixed unital $C^*$-algebra $A$. We do require that the unit $1$ from $A$ belongs to $\C X$ and of course that the involution from $A$ preserves $\C X$. The operator system $\C X$ inherits both a norm and a notion of positive elements from $A$. We identify the complex numbers $\B C$ with the subspace $\B C \cd 1 \su \C X$. 


\section{A $q$-deformed differential calculus on the unit circle}\label{s:diffcalc}
Throughout this section we fix a $q$ in $(0,1]$. For every integer $n \in \B Z$, the corresponding $q$-integer is defined by
\[
[n]_q := \fork{ccc}{ \frac{q^n - q^{-n}}{q - q^{-1}} & \T{for} & q \in (0,1) \\ n & \T{for} & q = 1} . 
\]

Let $\C O(S^1)$ denote the smallest unital $\ast$-subalgebra of the unital \(C^*\)-algebra \(C(S^1)\) containing the inclusion $z : S^1 \to \B C$. As a vector space $\C O(S^1)$ has a basis given by the integer powers of $z$, thus $( z^n)_{n\in \B Z}$.

The aim of this section is to construct a $q$-deformed first order differential calculus (FODC) over $\C O(S^1)$. In the case where $q = 1$, we recover an algebraic version of the usual de Rham complex for the unit circle (tensorized with the complex numbers). By a first order differential calculus, we mean the following, see e.g. \cite[12.1.1 Definition 1 and 2]{KLSC} or \cite[Definition 1.1 and 1.4]{BeMa:QRG}.

\begin{definition}
  A first order differential calculus over a unital algebra \(\C A\) consists of an \(\C A\)-bimodule \(\C M\) with a linear map \( \dif : \C A\to \C M\) such that
  \begin{enumerate}[i)]
    \item (Leibniz rule) \(\dif(ab)=a \triangleright \dif(b)+\dif(a)\triangleleft b\) for all \(a,b\in \C A\),
    \item \(\C M=\T{span}_{\B C}\big\{a\triangleright \dif(b)\triangleleft c \mid a,b,c\in \C A \big\}\).
  \end{enumerate}

In the case where \(\C A\) is a unital $\ast$-algebra, a FODC $(\dif,\C M)$ over \(\C A\) is a \emph{$\ast$-calculus}, if there exists an involution of the vector space \(\C M\) such that \((a\triangleright \dif(b)\triangleleft c)^*=c^*\triangleright \dif(b^*)\triangleleft a^*\) for all \(a,b,c\in \C A\).
\end{definition}

The unital $\ast$-algebra $\C O(S^1)$ comes equipped with an algebra automorphism  \(\sigma:\O(S^1)\to \O(S^1)\) determined by
\[
  \sigma(z^n)=q^nz^n \q \T{for all } n\in\Z.
  \]
  The inverse algebra automorphism \(\sigma^{-1}\) is given by \(z^n\mapsto q^{-n}z^n\). We record that $\si$ is not in general $\ast$-preserving. Instead it holds that
  \begin{equation}\label{eq:modular}
\sigma(f^*) = \si^{-1}(f)^* \q f \in \C O(S^1) .
\end{equation}

Put $\C X_q := \C O(S^1)$ and consider $\C X_q$ as a bimodule over $\C O(S^1)$ with left and right actions given by 
\[
  f\triangleright \xi = \sigma(f)\cdot \xi \q \xi\triangleleft f = \xi\cdot \sigma^{-1}(f),
\] 
where \(\cdot\) denotes the usual product in \(\O(S^1)\). The bimodule $\C X_q$ also inherits an involution $\ast$ from $\C O(S^1)$. 

\begin{definition}
The \emph{$q$-derivative} is the linear map \( \dif_q:\O(S^1)\to \C X_q\) defined by \(z^n\mapsto i[n]_qz^n\). 
\end{definition}

%

\begin{proposition}
The pair \((\X_q,\dif_q)\) is a $\ast$-calculus over \(\O(S^1)\).
\end{proposition}
\begin{proof}
  To verify the Leibniz rule we focus on the case where $a = z^n$ and $b = z^m$ for some $n,m \in \B Z$. In this situation, we use the identity $[n + m]_q = q^n [m]_q + [n]_q q^{-m}$ to obtain that 
  \[
\dif_q(z^{n+m}) = i q^n [m]_q z^{n+m} + i [n]_q q^{-m} z^{n+m} = \si(z^n) \cd \dif_q(z^m) + \dif_q(z^n) \cd \si^{-1}(z^m) .
\]

We now turn our attention to verifying condition $ii)$. It suffices to show that
\begin{equation}\label{eq:generatorspan}
  z^n\in\T{span}_{\B C}\big\{ f\triangleright \dif_q(g)\triangleleft h \mid f,g,h\in\O(S^1) \big\} \q \T{for all } n\in\Z,
\end{equation}
since the set af all such functions spans \(\X_q\). For $n \neq 0$ the relevant inclusion follows since
$z^n = 1 \triangleright \dif_q\big(z^n \frac{1}{i [n]_q} \big) \triangleleft  1$. For $n = 0$ the inclusion follows by observing that
$1 = -i q z^{-1}\triangleright\dif_q(z) \triangleleft 1$.

It remains to consider the interplay between the various involutions. To this end, we first notice that $\dif_q(x^*) = \dif_q(x)^*$ for all $x \in \O(S^1)$. The relevant identity now follows from \eqref{eq:modular}. 
\end{proof}

In the special case where $q = 1$, the $1$-derivative agrees with the vector field $\frac{d}{d\te} : \C O(S^1) \to \C O(S^1)$ given locally by differentiation with respect to the chart $S^1 \sem \{1\} \cong (0,2\pi)$ (mapping $e^{i \te}$ to $\te$). Indeed, we have that $\frac{d z^n}{d\te} = i n z^n = \dif_1(z^n)$ for all $n \in \B Z$ and our FODC therefore recovers the classical de Rham complex upon tensorizing with the complex numbers and replacing the smooth functions and $1$-forms with their algebraic counterparts.






\section{Extending the $q$-derivative}\label{sec:extending}
Recall that $q$ is a fixed number in $(0,1]$. The aim of this section is to extend the $q$-derivative $\dif_q : \O(S^1) \to \X_q$ enabling us to take the $q$-derivative of a much larger variety of functions. The most effective way of achieving this is to work at the level of bounded operators on $\ell^2(\B Z)$. In fact, our extended $q$-derivative is an example of an unbounded Schur multiplier. The related theory of bounded Schur multipliers is well-developed and the reader can for instance consult the chapters \cite[Chapter 5]{Pis:SPC}, \cite[Chapter 8]{paulsen} or the survey \cite{ToTu:SOM}. 

For a bounded operator $T : \ell^2(\B Z) \to \ell^2(\B Z)$, apply the notation $T_{jk} := \inn{e_j, T e_k}$ for the matrix units indexed by $j,k \in \B Z$. For every finite non-empty subset $F \su \B Z$ we then obtain the finite matrix $(T_{jk})_{j,k \in F}$ which can be viewed as an element in $M_{|F|}(\B C)$ and hence has a $C^*$-norm $\big\| (T_{jk})_{j,k \in F} \big\|_\infty$. Since $T$ is bounded it holds that the supremum
\[
\sup\Big\{ \big\| (T_{jk})_{j,k \in F} \big\|_\infty \mid F \su \B Z \, , \, \, \T{finite and non-empty}\Big\}
\]
is finite and it agrees with the operator norm of $T$. In fact, we may identify the bounded operators $\B B\big( \ell^2(\B Z) \big)$ with the infinite matrices $(M_{jk})_{j,k \in \B Z}$ (with entries in $\B C$) satisfying that the subset below is bounded from above: 
\[
\Big\{ \big\| (M_{jk})_{j,k \in F} \big\|_\infty \mid F \su \B Z \, , \, \, \T{finite and non-empty} \Big\} \su [0,\infty) .
\]

The notation $\ell^\infty(\B Z)$ refers to the unital $C^*$-algebra of bounded $\B C$-valued functions on $\B Z$ and we identify $\ell^\infty(\B Z)$ with the unital $C^*$-subalgebra of $\B B\big(\ell^2(\B Z)\big)$ consisting of diagonal infinite matrices. This identification maps a function $f \in \ell^\infty(\B Z)$ to the infinite matrix $( f(j) \de_{jk} )_{j,k \in \B Z}$, where $\de_{jk} \in \{0,1\}$ is the Kronecker delta. 

\begin{definition}\label{d:unbschur}
  Let $\mu : \B Z \ti \B Z \to \B C$ be a function. The \emph{unbounded Schur multiplier} associated to $\mu$ is the linear map
  $\B M(\mu) : \T{Dom}\big( \B M(\mu) \big) \to \B B\big(\ell^2(\B Z)\big)$ given by the formula
\[
\B M(\mu)\big( (M_{jk})_{j,k \in \B Z} \big) = \big( \mu(j,k) \cd M_{jk} \big)_{j,k \in \B Z}
\]
on the domain $\T{Dom}\big( \B M(\mu) \big)$ consisting of those bounded operators $(M_{jk})_{j,k \in \B Z}$ where $\big( \mu(j,k) \cd M_{jk} \big)_{j,k \in \B Z}$ is likewise a bounded operator.
\end{definition}

Remark that for a function $\mu : \B Z \ti \B Z \to \B C$, the corresponding unbounded Schur multiplier
$\B M(\mu) : \T{Dom}\big( \B M(\mu) \big) \to \B B\big(\ell^2(\B Z)\big)$ is automatically a closed unbounded operator. It therefore follows from the closed graph theorem that, if the domain of $\B M(\mu)$ agrees with $\B B\big(\ell^2(\B Z) \big)$, then the corresponding linear map
\[
\B M(\mu) : \B B\big( \ell^2(\B Z) \big) \to \B B\big( \ell^2(\B Z) \big)
\]
is automatically bounded. In this case, we refer to $\B M(\mu)$ as a \emph{bounded Schur multiplier}. In fact, by \cite[Theorem 5.1]{Pis:SPC} the bounded operator $\B M(\mu)$ is automatically completely bounded and the cb-norm of $\B M(\mu)$ agrees with the operator norm of $\B M(\mu)$. 

Let us return to the general setting where $\B M(\mu)$ is an unbounded Schur multiplier. We may consider the bounded operators $\B B\big( \ell^2(\B Z) \big)$ as a bimodule over $\ell^\infty(\B Z)$ with left and right action given by multiplication of bounded operators. In this fashion, it holds that the domain $\T{Dom}\big( \B M(\mu) \big) \su \B B\big(\ell^2(\B Z)\big)$ is a sub-bimodule over $\ell^\infty(\B Z)$ and the unbounded Schur multiplier $\B M(\mu)$ is bilinear over $\ell^\infty(\B Z)$. 

The proof of the next lemma is an elementary exercise in unbounded Schur multipliers. Notice in this respect that for a bounded operator $T : \ell^2(\B Z) \to \ell^2(\B Z)$ we get that $(T^*)_{jk} = \ov{T_{kj}}$ for all $j,k \in \B Z$.

\begin{lemma}\label{l:opsys}
  If the function $\mu : \B Z \ti \B Z \to \B C$ satisfies that $\ov{ \mu(k,j) } = \mu(j,k)$ for all $j,k \in \B Z$ and the supremum
  $\sup_{j \in \B Z} \big| \mu(j,j) \big|$ is finite, then the domain $\T{Dom}\big( \B M(\mu) \big) \su \B B\big(\ell^2(\B Z)\big)$ is an operator system. It moreover holds that $\B M(\mu)(T^*) = \big( \B M(\mu)(T) \big)^*$ for all $T \in \T{Dom}\big( \B M(\mu) \big)$.
\end{lemma}

Define the function $\de_q : \B Z \ti \B Z \to \B C$ by putting
\[
\de_q(j,k) := i \cd [j - k]_q \q \T{for all } j,k \in \B Z .
\]

The next lemma asserts that the unbounded Schur multiplier $\B M(\de_q) : \T{Dom}\big( \B M(\de_q) \big) \to \B B\big(\ell^2(\B Z)\big)$ extends the $q$-derivative from Section \ref{s:diffcalc}. As described in Section \ref{s:conventions} we identify $\C O(S^1)$ with a unital $*$-subalgebra of $\B B\big(\ell^2(\B Z)\big)$ by sending $z$ to the bilateral shift $U$.

\begin{lemma}\label{l:extension}
It holds that $\C O(S^1) \su \T{Dom}\big( \B M(\de_q) \big)$ and we have the identity $\B M(\de_q)(f) = \dif_q(f)$ for all $f \in \C O(S^1)$.
\end{lemma}
\begin{proof}
  Let $n \in \B Z$ and record that the infinite matrix associated with $U^n$ satisfies that $U^n_{n+k,k} = 1$ for all $k \in \B Z$ and all other entries are equal to zero. Since $\de_q(n+k,k) = i \cd [n]_q$ for all $k \in \B Z$ we get that $U^n \in \T{Dom}\big( \B M(\de_q) \big)$ with
  \[
\B M(\de_q)(U^n) = i \cd [n]_q \cd U^n . \qedhere
\]
\end{proof}

Because of Lemma \ref{l:extension} we apply the notation $\dif_q := \B M(\de_q)$ and refer to this Schur multiplier as the \emph{$q$-derivative}. Remark that Lemma \ref{l:opsys} entails that the domain $\T{Dom}\big( \B M(\de_q) \big) \su \B B(\ell^2(\B Z))$ is an operator system and that $\dif_q(T^*) = \dif_q(T)^*$ for all $T \in \T{Dom}\big( \B M(\de_q) \big)$. 
%
%

\begin{definition}
  The \emph{$q$-Lipschitz operators} is the domain of the $q$-derivative and this operator system is denoted by
  \[
\Lip_q\big( \ell^2(\B Z)\big) := \T{Dom}\big( \B M(\de_q) \big) .
\]
  The \emph{$q$-Lipschitz functions} $\Lip_q(S^1) \su C(S^1)$ is the operator system defined as the intersection
  \[
\Lip_q(S^1) := C(S^1) \cap \Lip_q\big( \ell^2(\B Z)\big) .
  \]
\end{definition}

Let us briefly investigate the situation where $q = 1$.

\begin{proposition}\label{p:oneLipschitz}
The $1$-Lipschitz functions $\Lip_1(S^1)$ agrees with the Lipschitz functions associated to the standard arc length metric on the unit circle. Moreover, for $f \in \Lip_1(S^1)$ it holds that the operator norm of the derivative $\| \dif_1(f) \|_\infty$ agrees with the associated Lipschitz constant for the Lipschitz function $f$.
\end{proposition}
\begin{proof}
  Define the strongly continuous action $\si$ of the real line on $C(S^1)$ by putting $\si_t(f)(z) := f(e^{-it} \cd z)$ for all $t \in \B R$ and $z \in S^1$.

  For a continuous function $f : S^1 \to \B C$, it can be verified that $f$ is a Lipschitz function for the standard metric on $S^1$ if and only if the supremum
  \begin{equation}\label{eq:suplip}
\sup_{t \in \B R \sem \{0\}} \frac{\| \si_t(f) - f\|_\infty}{|t|}
\end{equation}
is finite. And in this case, the above supremum agrees with the associated Lipschitz constant.

Let us introduce the selfadjoint unbounded operator $N : \T{Dom}(N) \to \ell^2(\B Z)$ determined by the formula $N(e_k) = k \cd e_k$ on the core given by the algebraic linear span of the basis vectors $e_k$ for $k \in \B Z$. We say that a continuous function $f : S^1 \to \B C$ is a Lipschitz operator with respect to $N$, if the following holds
\begin{itemize}
\item $f$ preserves the domain of $N$ and the commutator $[N,f] : \T{Dom}(N) \to \ell^2(\B Z)$ extends to a bounded operator on $\ell^2(\B Z)$.
\end{itemize}
For $t \in \B R$ it holds that the unitary operator $e^{i t N}$ (formed using the functional calculus for selfadjoint unbounded operators) is given by the formula $e^{it N}(e_k) = e^{it k} \cd e_k$ for $k \in \B Z$. Consequently, $N$ induces the strongly continuous action $\si$ in so far that
\[
\si_t(f) = e^{-itN} f e^{it N}  \q \T{for all } f \in C(S^1) .
\]

For a continuous function $f : S^1 \to \B C$, it is therefore a consequence of \cite[Theorem 3.8]{Chr:WDO} that the supremum in \eqref{eq:suplip} is finite if and only if $f$ is a Lipschitz operator with respect to $N$. In this case, the supremum in \eqref{eq:suplip} agrees with the operator norm of the bounded extension of the commutator $[N,f] : \T{Dom}(N) \to \ell^2(\B Z)$.

For every $T \in \B B\big(\ell^2(\B Z)\big)$ and $j,k \in \B Z$, we now remark that
\[
-i \cd \de_1(j,k) T_{jk} = (j - k) T_{jk} = \inn{N e_j, T e_k} - \inn{e_j, T N e_k} .
\]

Applying \cite[Theorem 3.8]{Chr:WDO} one more time, we therefore get that a continuous function $f : S^1 \to \B C$ is a $1$-Lipschitz function if and only if $f$ is a Lipschitz operator with respect to $N$. In this situation, $-i \cd \dif_1(f)$ agrees with the bounded extension of the commutator $[N,f] : \T{Dom}(N) \to \ell^2(\B Z)$. 

This proves the result of the proposition. 
\end{proof}






    \section{\(q\)-integration}\label{sec:qintegral}
    We remind the reader that \(q\in (0,1]\) is a fixed number. Our goal is now to construct a $q$-analogue of the classical Riemann integral for continuous functions on the unit circle $S^1$ equipped with the Haar measure. In our language, this means that we are looking for an inverse to the $q$-derivative $\dif_q : \C O(S^1) \to \C X_q$ up to addition of constant functions. In line with the investigations carried out in Section \ref{sec:extending} it is natural for us to consider a much more general question, namely whether we can find an inverse to the extended $q$-derivative $\dif_q$ up to addition of elements from the kernel of $\dif_q$. The content of this section relies on the results obtained in \cite[Section 5.4 and 5.5]{JKDK_QSU2} where a $q$-integral is constructed in a different context by means of Schur multipliers. This means that the properties of our $q$-integral are derived from Grothendieck's theorem on Schur multipliers, see e.g. \cite[Theorem 5.1]{Pis:SPC} and \cite[Corollary 8.8]{paulsen}. 

We start out by computing the kernel of the extended $q$-derivative $\dif_q : \Lip_q\big( \ell^2(\B Z) \big) \to \B B\big(\ell^2(\B Z)\big)$.

   \begin{lemma}
It holds that $\ker(\dif_q) = \ell^\infty(\B Z)$.
    \end{lemma}
    \begin{proof}
This follows immediately by noting that $\de_q(j,k) = 0$ if and only if $j = k$.
    \end{proof}

Our $q$-integral will be the Schur multiplier associated with the function $\psi_q : \B Z \ti \B Z \to \B C$ defined by
\[
\psi_q(j,k) := \fork{ccc}{ -i \cd \frac{1}{[j-k]_q} & \T{for} & j \neq k \\ 0 & \T{for} & j = k} .
\]
Slight modifications to the proof of \cite[Lemma 5.4.3]{JKDK_QSU2} yields the following key result. Notice in this respect that $[n]_q  \geq n$ for all $n \in \B N_0$ (see the proof of \cite[Lemma 5.4.2]{JKDK_QSU2}). 

    \begin{lemma}\label{l:intesti}
 The Schur multiplier $\B M( \psi_q)$ is bounded and we have the estimate 
$\big\| \B M(\psi_q) \big\|_{cb} \leq \frac{\pi}{\sqrt{3}}$ on the corresponding cb-norm.
   \end{lemma}

 We apply the suggestive notation $\int_q := \B M(\psi_q)$ and refer to $\int_q : \B B\big( \ell^2(\B Z) \big) \to \B B\big( \ell^2(\B Z) \big)$ as the $q$-integral.


    In order to understand the relationship between the $q$-derivative and the $q$-integral we introduce the conditional expectation $E : \B B\big(\ell^2(\B Z)\big) \to \B B\big(\ell^2(\B Z)\big)$ given by
\[
E\big( (T_{jk})_{j,k \in \B Z} \big) := (T_{jk} \de_{jk})_{j,k \in \B Z}
\]
and record that the image of $E$ agrees with $\ell^\infty(\B Z)$. Remark that $E$ is the bounded Schur multiplier associated with the Kronecker delta function $\de : \B Z \ti \B Z \to \B C$.


\begin{proposition}\label{p:fundamental}
  We have the identities
  \[
T - \int_q \dif_q(T) = E(T) \, \, \mbox{ and } \, \, \, S - \dif_q \int_q S = E(S)
\]
for all $T \in \Lip_q\big( \ell^2(\B Z) \big)$ and $S \in \B B\big(\ell^2(\B Z)\big)$. In particular, it holds that the image of $\int_q$ is contained in the domain of $\dif_q$.
\end{proposition}
\begin{proof}
  This is a consequence of the identity
  \[
1 - \psi_q(j,k) \cd \de_q(j,k) = \de_{jk} \q j,k \in \B Z
\]
and basic properties of Schur multipliers.
\end{proof}

\begin{corollary}
The image of the $q$-derivative $\dif_q : \Lip_q\big( \ell^2(\B Z) \big) \to \B B\big(\ell^2(\B Z)\big)$ agrees with $\ker(E)$ and the image of the $q$-integral $\int_q : \B B\big(\ell^2(\B Z)\big) \to \B B\big(\ell^2(\B Z)\big)$ agrees with the intersection $\Lip_q\big( \ell^2(\B Z) \big) \cap \ker(E)$.
\end{corollary}

\begin{proposition}\label{p:conditional}
  We have the estimate
  \[
\| T - E(T) \|_\infty \leq \frac{\pi}{\sqrt{3}} \cd \| \dif_q(T) \|_\infty \q \mbox{for all } T \in \Lip_q\big( \ell^2(\B Z) \big) .
\]
\end{proposition}
\begin{proof}
  Let $T$ be a $q$-Lipschitz operator. An application of Lemma \ref{l:intesti} and Proposition \ref{p:fundamental} entails that
  \[
\| T - E(T) \|_\infty  = \| \int_q \dif_q(T) \|_\infty \leq \frac{\pi}{\sqrt{3}} \cd \| \dif_q(T) \|_\infty . \qedhere
  \]
\end{proof}

We end this section by specializing parts of our results to statements regarding the unit circle.

First of all, notice that the $q$-integral restricts to a bounded operator $\int_q : C(S^1) \to C(S^1)$ satisfying that $\int_q z^n = -i \cd \frac{1}{[n]_q} \cd z^n$ for $n \in \B Z \sem \{0\}$ and $\int_q 1 = 0$. It moreover holds that $E(U^n) = 0$ for $n \neq 0$ whereas $E(1) = 1$. Thus, letting $h : C(S^1) \to \B C$ denote the Haar state associated to the unit circle (so that $h$ is given by integration with respect to the normalized arc length measure) we obtain that $E(f) = h(f) \cd 1$ for all $f \in C(S^1)$.


\begin{proposition}\label{p:findiam}
  The kernel of the $q$-derivative $\dif_q : \Lip_q(S^1) \to \B B\big( \ell^2(\B Z)\big)$ agrees with the constant functions and for every $f \in \Lip_q(S^1)$ we have the identity
  \[
f - \int_q \dif_q(f) = h(f) \cd 1 
\]
as well as the estimate $\| f - h(f) \cd 1 \|_\infty \leq \frac{\pi}{\sqrt{3}} \cd \| \dif_q(f) \|_\infty$.
\end{proposition}






\section{Quantum metrics derived from the $q$-derivative}\label{sec:cqms}
We apply the extended $q$-derivative from Section \ref{sec:extending} to construct a seminorm $L_q$ on the operator system of $q$-Lipschitz operators, defining
\[
L_q(T) := \big\| \dif_q(T) \big\|_\infty \q \T{for all } T \in \Lip_q\big( \ell^2(\B Z) \big) .
\]
In this section, we investigate the properties of the pair $\big( \Lip_q\big( \ell^2(\B Z) \big), L_q \big)$ in light of Rieffel's theory of compact quantum metric spaces, see \cite{Rie:MSA,Rie:MSS}. 

The main result of this section is that the seminorm $L_q$ provides the $q$-Lipschitz functions $\Lip_q(S^1)$ with the structure of a compact quantum metric space. 

\subsection{Some background on compact quantum metric spaces}
Let us review the operator system approach to compact quantum metric spaces as introduced by Kerr in \cite{Ker:MQG}. The following concept of a \emph{slip-norm} plays a key role: 

\begin{definition}
  Let \(\X \) be an operator system. We say that a seminorm  \(L:\X\to[0,\infty)\) is a \emph{slip-norm}, if
  \begin{itemize}
  \item \(L(x)=L(x^*)\) for all \(x\in\X\) and $L(1) = 0$.
  \end{itemize}
\end{definition}

Recall that our (concrete) operator system $\X$ sits inside a unital $C^*$-algebra $A$. We apply the notation $X \su A$ for the complete operator system obtained by taking the norm-closure of $\X \su A$.

Being a compact quantum metric space is a characterization of how a slip-norm \(L : \X \to [0,\infty)\) interacts with the \emph{state space} of \(\X\), which we denote by \(S(\X)\). The state space $S(\C X)$ consists of the unital positive linear functionals on $\C X$. Notice that the restriction map induces a homeomorphism $S(X) \cong S(\C X)$. 

\begin{definition}
  Let \(L:\X\to[0,\infty)\) be a slip-norm. The \emph{Monge-Kantorovich metric} \(\rho_L:S(\X)\times S(\X)\to[0,\infty]\) associated with \(L\) is defined by
  \[
    \rho_L(\phi,\psi)=\sup\big\{|\phi(x)-\psi(x)| \mid L(x)\le 1 \big\}.
\]
\end{definition}

Except for the fact that the Monge-Kantorovich metric may take the value \(\infty\), it satisfies the axioms of a metric. We therefore obtain a topology on \(S(\X)\) by taking as a basis all the open balls (with respect to $\rho_L$).

\begin{definition}
  An operator system \(\X\) with slip-norm \(L\) is a \emph{compact quantum metric space}, if the metric topology induced by $\rho_L$ agrees with the weak-$\ast$\ topology on \(S(\X)\). In this case, we also say that $(\C X,L)$ is a quantum metric on $X$. 
\end{definition}

If the pair \((\X,L)\) is a compact quantum metric space, then \(\rho_L\) is finite due to the weak-$\ast$\ compactness and convexity of \(S(\X)\). Hence, in this situation $\rho_L$ is a metric in the usual sense.

We now define what it means to have finite diameter. To this end, let $[ \cd ] : \C X \to \C X / \B C$ denote the quotient map and equip the quotient space $\C X/\B C$ with the quotient norm $\| \cd \|_{\X/\B C}$.

\begin{definition}\label{d:findiam}
  Let \(L:\X\to[0,\infty)\) be a slip-norm. We say that the pair \((\X,L)\) has \emph{finite diameter}, if there exists an \(r>0\) such that
  \[
    \big\| [x] \big\|_{\X/\B C}\le r \cd L(x) \q \mbox{for all } x\in\X.
    \]
\end{definition}

We clarify that the pair $(\C X,L)$ has finite diameter in the sense of Definition \ref{d:findiam} if and only if the state space has finite diameter with respect to the Monge-Kantorovich metric $\rho_L$, see \cite[Proposition 1.6]{Rie:MSA}. In this case, we apply the notation $\T{diam}(\C X,L)$ for the diameter of the metric space $\big( S(\C X), \rho_L\big)$. Moreover, if $(\C X,L)$ has finite diameter it clearly holds that the kernel of $L$ agrees with the scalars $\B C$.

It is convenient for us to apply the characterization of compact quantum metric spaces in terms of finite dimensional approximations found as \cite[Theorem 3.1]{jk23}. We therefore review how this works.

\begin{definition}\label{cqms:def:approx}
  Let \(L : \X \to [0,\infty)\) be a slip-norm. Let further \(\epsilon,C>0\) be constants and let \(\mathcal Y\) be an operator system. We say that a pair \((\iota,\Phi)\) consisting of unital bounded operators \(\iota,\Phi:\X\to\mathcal Y\) is an \((\epsilon,C)\)-approximation of \((\X,L)\), if the following holds:
  \begin{enumerate}
  \item \(\frac{1}{C}\|x\|_\infty \le \|\iota(x)\|_\infty\) for all \(x\in\X\);
  \item \(\Phi(\X)\) is finite dimensional;
    \item \(\|\iota(x)-\Phi(x)\|_\infty\le \epsilon \cd L(x)\) for all \(x\in\X\).
  \end{enumerate}
\end{definition}

Remark that (1) implies that the unital bounded operator \(\iota : \C X \to \C Y\) induces a linear bounded isomorphism onto its image and the corresponding inverse is also bounded (with operator norm less than or equal to $C$). 
\begin{theorem}\label{cqms:thm:approx}
  Let \(L:\X\to[0,\infty)\) be a slip-norm. The following conditions are equivalent.
  \begin{enumerate}
  \item \((\X,L)\) is a compact quantum metric space;
    \item \((\X,L)\) has finite diameter and there exists a constant \(C>0\) such that for every \(\epsilon>0\) there exists an \((\epsilon,C)\)-approximation of \((\X,L)\).
  \end{enumerate}
\end{theorem}

Let us emphasize that if $(\C X,L)$ is a compact quantum metric space and $\C Z \su \C X$ is an operator subsystem, then the pair $(\C Z,L)$ is automatically a compact quantum metric space as well. This can for instance be seen from Theorem \ref{cqms:thm:approx} or from \cite[Theorem 1.8]{Rie:MSA}.

\subsection{Quantum metrics on the circle}
Let $q \in (0,1]$ and consider the seminorm $L_q : \Lip_q(S^1) \to [0,\infty)$ which we defined in the very beginning of the present section. It clearly holds that $L_q$ is a slip-norm and we are now going to see that the pair \((\Lip_q(S^1),L_q)\) is a quantum metric on $C(S^1)$. First of all, our results on the $q$-integral from Section \ref{sec:qintegral} enable us to verify the finite diameter condition.

\begin{proposition}\label{cqms:prop:fdfoxonly}
 The pair \((\Lip_q(S^1),L_q)\) has finite diameter. In fact, it holds that
  \[
\| f \|_{\Lip_q(S^1)/\B C} \leq \frac{\pi}{ \sqrt{3}} \cd L_q(f) \q \mbox{for all } f \in \Lip_q(S^1)
\]
and we therefore have the estimate $\T{diam}\big( (\Lip_q(S^1),L_q) \big) \leq \frac{2\pi}{\sqrt{3}}$.
\end{proposition}
\begin{proof}
  This follows immediately from Proposition \ref{p:findiam} and \cite[Proposition 1.6]{Rie:MSA}.
\end{proof}

It is relevant to remark that the result of Proposition \ref{p:conditional} actually yields the following estimate regarding $q$-Lipschitz operators instead of $q$-Lipschitz functions:
\[
\| T \|_{\Lip_q(\ell^2(\B Z))/ \ell^\infty(\B Z)} \leq \frac{\pi}{ \sqrt{3}} \cd L_q(T)
\q \T{for all } T \in \Lip_q\big(\ell^2(\B Z)\big) .
\]

In order to construct finite dimensional approximations of the pair \((\Lip_q(S^1),L_q)\) in line with Definition \ref{cqms:def:approx}, we introduce another Schur multiplier. 

Let \(M\in\N_0\) and define \(\gamma_M:\Z \ti \Z \to\B C\) by
\[
  \gamma_M(j,k):=\fork{ccc}{
    \frac{M+1-\abs{j-k}}{M+1} & \T{for} &  \abs{j-k}\le M \\
    0 & \T{for} & \abs{j-k}>M } . 
\]
The following lemma can be found as \cite[Lemma 5.4.4]{JKDK_QSU2}. 

\begin{lemma}\label{l:schurapp}
  The Schur multiplier $\B M(\ga_M)$ is bounded and we have the estimate $\big\|\mathbb M(\gamma_M)\big\|_{cb}\le 1$ on the corresponding cb-norm.
\end{lemma}

Remark that $\B M(\ga_M)$ is unital since $\ga_M(j,j) = 1$ for all $j \in \B Z$. In the next lemma, the image of $\B M(\ga_M)$ is computed.

\begin{lemma}\label{l:imagemodule}
  The bounded Schur multiplier $\B M(\ga_M)$ is bilinear over $\ell^\infty(\B Z)$ and the image is given by
  \begin{equation}\label{eq:image}
  \B M(\ga_M)\big( \B B(\ell^2(\B Z)) \big)
  = \T{span}_{\B C} \big\{ U^n f \mid |n| \leq M \, \, , \, \, \, f \in \ell^\infty(\B Z) \big\} .
  \end{equation}
  In particular, the image of $\B M(\ga_M)$ is isomorphic as a right module (and as a left module) to the free module $\ell^\infty(\B Z)^{\op (2M+1)}$.
\end{lemma}
\begin{proof}
  The bilinearity of $\B M(\ga_M)$ follows since $\B M(\ga_M)$ is a Schur multiplier, see the discussion after Definition \ref{d:unbschur}.
For $n \in \B Z$, observe that $\B M(\ga_M)(U^n) = 0$ for $|n| > M$ and $\B M(\ga_M)(U^n) = \frac{M+1 - |n|}{M+1} \cd U^n$ for $|n| \leq M$. These observations imply the identity in \eqref{eq:image}. The isomorphism between the free module and the image of $\B M(\ga_M)$ is provided by the right module map $\{ f_n \}_{n = -M}^M \mapsto \sum_{n = -M}^M U^n f_n$. For the case of left modules over $\ell^\infty(\B Z)$ the relevant isomorphism is given by $\{ f_n \}_{n = -M}^M \mapsto \sum_{n = -M}^M f_n U^n$ instead.  
\end{proof}

The argument provided in Lemma \ref{l:imagemodule} also allows us to compute the image of the restriction of $\B M(\ga_M)$ to $C(S^1)$:

\begin{lemma}\label{l:findim}
  We have the identity
  \begin{equation}\label{eq:imagecircle}
  \B M(\ga_M)\big( C(S^1) \big)
  = \T{span}_{\B C} \big\{ U^n \mid |n| \leq M \big\} .
  \end{equation}
  In particular, the image $\B M(\ga_M)\big( C(S^1) \big)$ is a finite dimensional subspace of $\C O(S^1)$.
\end{lemma}

Let us introduce the strictly positive number 
\[
  \epsilon_M :=\sqrt{2}\cdot \Big(\frac{M}{(M+1)^2}+\sum_{k=M+1}^\infty\frac{1}{k^2}\Big)^{\frac{1}{2}}
  \]
  and record that $\lim_{M \to \infty} \epsilon_M = 0$. 

The following proposition is a modification of \cite[Proposition 5.5.3]{JKDK_QSU2}. Comparing with \cite[Proposition 5.5.3]{JKDK_QSU2} we are systematically using $q$-integers instead of $q$-half-integers and this means that our estimates are slightly simpler. We remind the reader that $| [n]_q| \geq |n|$ for all $n \in \B Z$.

\begin{proposition}\label{cqms:prop:eapprox}
  Let \(T \in\Lip_q\big(\ell^2(\B Z)\big)\). It holds that
  \[
  \big\| T- \B M(\ga_M)(T) \big\|_\infty \le \epsilon_M \cdot L_q(T)
  \q \mbox{for all } M \in \B N_0 \, \, \mbox{ and } \, \, \, q \in (0,1] .
  \]
\end{proposition}

We are now ready to prove the first main theorem of this paper.

\begin{theorem}\label{cqms:thm:cqms}
  Let \(q\in (0,1]\). The pair \((\Lip_q(S^1),L_q)\) is a quantum metric on $C(S^1)$.
\end{theorem}
\begin{proof}
  We aim to use Theorem \ref{cqms:thm:approx} to prove the claim of the theorem. Specifically, we use the implication \((2)\Rightarrow (1)\). From Proposition \ref{cqms:prop:fdfoxonly} we have that the pair \((\Lip_q(S^1),L_q)\) has finite diameter. We show that there for all \(\epsilon>0\) exists an \((\epsilon,1)\)-approximation of \((\Lip_q(S^1),L_q)\).

  Let \(\epsilon>0\) be given. We let the identity map \(1:\Lip_q(S^1)\to \Lip_q(S^1)\) act as our \(\iota\) from Definition \ref{cqms:def:approx}. That this map satisfies condition $(1)$ of the definition is then obvious, since the identity map is an isometry.

   Since \(\lim_{M\to\infty}\epsilon_M = 0\), we can choose \(M_0\in\N_0\) such that \(\epsilon_M <\epsilon\) for all \(M\ge M_0\). We take the bounded Schur multiplier \(\B M(\ga_{M_0})\) to be our \(\Phi\) corresponding to Theorem \ref{cqms:def:approx}. As a consequence of Proposition \ref{cqms:prop:eapprox}, this choice of \(\iota\) and \(\Phi\) satisfies condition $(3)$ of Definition \ref{cqms:def:approx}. The remaining condition $(2)$ follows immediately from Lemma \ref{l:findim}.
\end{proof}

\subsection{Non-existence of certain metrics on the circle}
The final aim of this section is to show that, for $q \in (0,1)$, the quantum metric \((\Lip_q(S^1),L_q)\) on $C(S^1)$ found in Theorem \ref{cqms:thm:cqms} does not come from a metric on the circle $S^1$. To explain what is meant by this statement let us consider the general case where $M$ is a compact metric space with metric $\rho$. We let $C(M)$ denote the unital $C^*$-algebra of continuous complex-valued functions on $M$ equipped with the supremum norm $\| \cd \|_\infty$. The norm-dense unital $*$-subalgebra of Lipschitz functions on $M$ is denoted by $\Lip_\rho(M) \su C(M)$ and we equip $\Lip_\rho(M)$ with the seminorm $L_\rho$ which associates the Lipschitz constant to a Lipschitz function $f : M \to \B C$. The prototypical example of a compact quantum metric space is given by the pair $\big( \Lip_\rho(M), L_\rho\big)$ and in this setting, the following \emph{Leibniz inequality} holds:
\[
L_\rho(fg)\le \norm{f}_\infty L_\rho(g)+L_\rho(f)\norm{g}_\infty \q \T{for all } f,g \in \Lip_\rho(M) . 
\]

\begin{definition}
Let $(\C X,L)$ be a quantum metric on $C(M)$. We say that $(\C X,L)$ \emph{comes from the metric} $\rho$, if $\C X$ is an operator subsystem of $\Lip_\rho(M)$ and $L(f) = L_\rho(f)$ for all $f \in \C X$.  
\end{definition}

We may now restate the content of Proposition \ref{p:oneLipschitz} and part of Theorem \ref{cqms:thm:cqms} in the following way: 

\begin{proposition}
The pair $(\Lip_1(S^1),L_1)$ is a quantum metric on $C(S^1)$ and this quantum metric comes from the standard arc length metric on $S^1$. It moreover holds that $\Lip_1(S^1)$ agrees with the Lipschitz functions coming the arc length metric.
\end{proposition}

The situation is different for $q \neq 1$ as clarified in the next lemma. 


\begin{lemma}\label{l:nonleibniz}
  Let \(q\in (0,1)\). It is impossible to find a constant $C > 0$ such that
  \begin{equation}\label{eq:leib}
L_q(f g) \leq C \cd \big( L_q(f) \| g \|_\infty + \| f \|_\infty L_q(g) \big) \q \mbox{for all } f,g \in \C O(S^1) .
  \end{equation}
\end{lemma}
\begin{proof}
  Suppose for contradiction that \eqref{eq:leib} is true for some \(C>0\). In particular, we have that 
  \[
   [2n]_q = L_q(z^{2n})\leq 2C\cdot L_q(z^n) = 2C \cd [n]_q \q \T{for all } n \in \B N .
   \]
   This entails that
   \[
   q^{-n} - q^{3n} =  q^n(q^{-1} - q) \cd [2n]_q \leq 2C \cd q^n(q^{-1} - q)[n]_q = 2C \cd (1 - q^{2n}) \q \T{for all } n \in \B N .
   \]
   But this is a contradiction since the left hand side diverges to infinity and the right hand side converges to $2C$ as $n$ approaches infinity.
\end{proof}

A combination of Lemma \ref{l:nonleibniz} and the fact that the Leibniz inequality is satisfied by Lipschitz functions yield our result regarding non-existence of metrics. 

\begin{proposition}
  Let \(q\in(0,1)\). If $\C X \su \Lip_q(S^1)$ is an operator subsystem with $\C O(S^1) \su \C X$, then the quantum metric $(\C X,L_q)$ on $C(S^1)$ does not come from a metric on $S^1$. 
\end{proposition}











\section{Continuity of the twist}
The result of Theorem \ref{cqms:thm:cqms} says that to each $q \in (0,1]$ we may assign the compact quantum metric space \((\Lip_q(S^1),L_q)\). In this section, we will show that this assignment is continuous, in the sense that \(\lim_{q\to q_0}\compdist\big( (\Lip_q(S^1),L_q),(\Lip_{q_0}(S^1),L_{q_0})\big)=0\) for all \(q_0\in(0,1]\), where \(\compdist\) refers to Kerr's complete Gromov-Hausdorff distance, see \cite{Ker:MQG}. In this endeavour, we employ a result due to Leimbach, see \cite[Proposition 5.19]{lb25} and the related results in \cite[Proposition 2.2.4]{JKDK_QSU2} and \cite[Theorem 5]{Sui:GHC}. 

\begin{theorem}\label{cont:thm:estimate}
  Suppose that \((\X,L_\X)\) and \((\Y,L_\Y)\) are compact quantum metric spaces and that \(\Phi:\X\to\Y\) and \(\Psi:\Y\to\X\) are ucp-maps such that there exist positive real numbers \(\epsilon_\X,\epsilon_\Y,C_\Phi,C_\Psi>0\) with the following properties:
  \begin{enumerate}
  \item \(L_\Y(\Phi(x))\le C_\Phi\cdot L_\X(x)\) and \(L_\X(\Psi(y))\le C_\Psi\cdot L_\Y(y)\) for all \(x\in\X\) and \(y\in\Y\);
    \item \(\|x-\Psi\Phi(x)\|_\X \le \epsilon_\X \cd L_\X(x)\) and \(\|y-\Phi\Psi(y)\|_\Y \le\epsilon_\Y \cd L_\Y(y)\) for all \(x\in\X\) and \(y\in\Y\).
  \end{enumerate}
  Then it holds that
  \[
  \begin{split}
    \compdist\big((\X,L_\X),(\Y,L_\Y)\big) 
    \le\max \Big\{ & \diam(\X,L_\X)\cdot\abs{1-1/C_\Phi}+\epsilon_\X/C_\Phi, \\
    & \q \diam(\Y,L_\Y)\cdot\abs{1-1/C_\Psi}+\epsilon_\Y/C_\Psi\Big\}.
  \end{split}
  \]
\end{theorem}

For every \(M\in\N_0\), we define the \emph{spectral band} as the finite dimensional operator subsystem
\[
  A_M:=\T{span}_{\B C}\big\{U^n \mid n\in\{-M,-M+1,\ldots,M\} \big\}\subseteq\O(S^1).
\]
Equipping it with the seminorm obtained by restricting \(L_q\) to \(A_M\), we obtain a compact quantum metric space.

Recall from Lemma \ref{l:schurapp} that the Schur multiplier \(\B M(\ga_M)\) is bounded and unital with $\| \B M(\gamma_M) \|_{cb} \le 1$. It therefore follows by \cite[Proposition 2.11]{paulsen} that $\B M(\ga_M)$ is a ucp-map. Furthermore, by Lemma \ref{l:findim}, the bounded Schur multiplier \(\B M(\ga_M)\) induces a surjective map \(C(S^1)\to A_M\). These observations along with Proposition \ref{cqms:prop:eapprox} and Theorem \ref{cont:thm:estimate} allow us to establish the following:

\begin{proposition}\label{cont:prop:uniform}
  For every \(q\in(0,1]\) and every \(M\in\N_0\), it holds that
  \begin{equation}\label{eq:spectralband}
    \compdist\big((\Lip_q(S^1),L_q),(A_M,L_q)\big) \le \ep_M
  \end{equation}
  and hence that $\compdist\big((\Lip_q(S^1),L_q),(\O(S^1),L_q)\big)=0$.
\end{proposition}
\begin{proof}
  Let \(M\in\N_0\) and \(q\in(0,1]\). The inclusion of $A_M$ into \(\Lip_q(S^1)\) is denoted by \(\iota: A_M\to\Lip_q(S^1)\) and we let $E_M : \Lip_q(S^1) \to A_M$ denote the ucp-map obtained by restriction of $\B M(\ga_M)$. It clearly holds that $\io$ is a ucp-map which satisfies that $L_q(\io(f)) = L_q(f)$ for all $f \in A_M$. Moreover, we obtain from Proposition \ref{cqms:prop:eapprox} that $\| f-E_M \io(f) \|_\infty  \le \ep_M \cdot L_q(f)$ for all $f \in A_M$ and likewise that $\| T-\io E_M(T) \|_\infty  \le \ep_M \cdot L_q(T)$ for all $T \in \Lip_q(S^1)$.

  Notice next that $\B M(\ga_M)$ and $\dif_q = \B M(\de_q)$ are both Schur multipliers with $\B M(\ga_M)$ being bounded and having image contained in the domain of $\dif_q$, see Lemma \ref{l:imagemodule}. It therefore holds that
  \begin{equation}\label{eq:schurcomm}
\dif_q \B M(\ga_M)(T) = \B M(\ga_M) \dif_q(T) \q \T{for all } T \in \Lip_q\big( \ell^2(\B Z)\big) .
\end{equation}
As a consequence of \eqref{eq:schurcomm} and since $E_M$ is a norm-contraction we get the estimate $L_q(E_M(T)) \leq L_q(T)$ for all $T \in \Lip_q(S^1)$. An application of Theorem \ref{cont:thm:estimate} now yields the inequality in \eqref{eq:spectralband}.

Finally, since $A_M \su \C O(S^1)$ for all $M \in \B N_0$ and $\lim_{M \to \infty}\ep_M = 0$ we may also conlude that
    \[
    \compdist\big((\Lip_q(S^1),L_q),(\C O(S^1),L_q)\big)=0. \qedhere
    \]
\end{proof}

For every $q \in (0,1]$ and every $M \in \B N_0$, it follows from the proof of Lemma \ref{l:extension} that the unbounded Schur multiplier \(\dif_q = \B M(\de_q)\) restricts to a bounded operator \((\dif_q)_M : A_M\to A_M\). We shall now see that these bounded operators vary continuously with respect to the deformation parameter \(q \in (0,1]\):
    
\begin{lemma}\label{cont:lemma:cont}
  Let \(M\in\N_0\). For every \(q_0\in(0,1]\), we have that
  \[
    \lim_{q\to q_0}(\dif_q)_M= (\dif_{q_0})_M ,
  \]
  where the limit is computed with respect to the operator norm on \(\B B( A_M,A_M)\).
\end{lemma}
\begin{proof}
  Let $q_0 \in (0,1]$ and $\ep > 0$ be given. Choose a $\de > 0$ such that $\big| [n]_q - [n]_{q_0} \big| < \ep/(2M+1)$ for all $q \in (0,1] \cap (q_0 - \de,q_0 + \de)$ and all $n \in \{-M,-M+1,\ldots,M\}$. For every $q \in (0,1] \cap (q_0 - \de,q_0 + \de)$ and every $f = \sum_{n=-M}^M \al_n U^n \in A_M$ with $\| f \|_\infty = 1$ it then holds that
        \[
    \begin{split}
      \big\| \dif_q(f)-\dif_{q_0}(f) \big\|_\infty
      = \big\| \sum_{n=-M}^M([n]_{q}-[n]_{q_0}) \al_n U^n  \big\|_\infty 
      < \sum_{n=-M}^M \frac{\ep}{2M+1} | \al_n | \leq \ep . \qedhere
    \end{split}
    \]
\end{proof}

The continuity result in Lemma \ref{cont:lemma:cont} allows us to prove a continuity result for the spectral bands with respect to Kerr's complete Gromov-Hausdorff distance:

\begin{proposition}\label{cont:prop:specband}
  For every \(M\in\N_0\) and every \(q_0\in(0,1]\), it holds that
  \[
    \lim_{q\to q_0}\compdist\big((A_M,L_q),(A_M,L_{q_0})\big) = 0.
  \]
\end{proposition}
\begin{proof}
  Let \(M\in\N_0\) and \(q_0\in(0,1]\) be given. Define the function $\chi : (0,1] \to [0,\infty)$ by putting $\chi(q) := \big\| (\dif_q)_M - (\dif_{q_0})_M \big\|$ for all $q \in (0,1]$. 

  Let $q \in (0,1]$. For every \(f\in A_M\) we get from Proposition \ref{p:findiam} that
    \[
  \begin{split}
    L_q(f)& \le \big\|(\dif_q)_M(f)-(\dif_{q_0})_M(f)\big\|_\infty + L_{q_0}(f)\\
          & =\big\|\big((\dif_q)_M-(\dif_{q_0})_M \big)\big(f-h(f) \cd 1\big)\big\|_\infty + L_{q_0}(f)\\
    & \le \chi(q) \cd \big\| f-h(f) \cd 1 \big\|_\infty + L_{q_0}(f)
    \le \big( \chi(q) \cd \pi/\sqrt{3} + 1\big)\cdot L_{q_0}(f) .
  \end{split}
  \]
  Likewise, we obtain that $L_{q_0}(f)\le \big( \chi(q) \cd \pi/\sqrt{3} + 1\big) \cd L_{q}(f)$ for all $f\in A_M$. 

      Since we know from Proposition \ref{cqms:prop:fdfoxonly} that \(\diam( A_M,L_q)\le \frac{2\pi}{\sqrt 3}\), an application of Theorem \ref{cont:thm:estimate} yields that
  \[
    \compdist\big((A_M,L_q),(A_M,L_{q_0})\big)\le \frac{2\pi}{\sqrt 3} \cd \Big| 1-\frac{1}{\chi(q) \cdot \pi/\sqrt{3}+1} \Big| .
  \]
The result of the present proposition now follows immediately from Lemma \ref{cont:lemma:cont}.
\end{proof}

Combining Proposition \ref{cont:prop:uniform} with Proposition \ref{cont:prop:specband}, we obtain the second and last main result of this paper:

\begin{theorem}
  For every \(q_0\in(0,1]\), it holds that
  \[
    \lim_{q\to q_0}\compdist\big((\Lip_q(S^1),L_q),(\Lip_{q_0}(S^1),L_{q_0})\big)=0.
  \]
\end{theorem}
\begin{proof}
  Let \(q_0\in(0,1]\) and \(\epsilon>0\) be given. By Proposition \ref{cont:prop:uniform} we may choose an \(M\in\N_0\) such that
$\compdist\big((A_M,L_q),(\Lip_q(S^1),L_q) \big) < \ep/3$ for all $q \in (0,1]$. Next, we may apply Proposition \ref{cont:prop:specband} to choose a \(\delta>0\) such that $\compdist\big((A_M,L_q),(A_M,L_{q_0}) \big) <\epsilon/3$ for all \(q\in (0,1]\) with \(|q-q_0|<\delta\). The triangle inequality for the complete Gromov-Hausdorff distance, \cite[Proposition 3.4]{Ker:MQG}, then entails that $\compdist\big((\Lip_q(S^1),L_q),(\Lip_{q_0}(S^1),L_{q_0})\big) < \ep$ for all \(q\in (0,1]\) with \(|q-q_0|<\delta\). This ends the proof of the theorem.
\end{proof}

\bibliographystyle{amsalpha-lmp}

\providecommand{\bysame}{\leavevmode\hbox to3em{\hrulefill}\thinspace}
\providecommand{\MR}{\relax\ifhmode\unskip\space\fi MR }
\providecommand{\MRhref}[2]{%
  \href{http://www.ams.org/mathscinet-getitem?mr=#1}{#2}
}
\providecommand{\href}[2]{#2}

\end{document}